\documentclass[a4paper,12pt]{article}
\usepackage[ansinew]{inputenc}
\usepackage{amsfonts,amsmath,amsthm}%
\usepackage{graphicx}
\usepackage[dvips]{color}

\newcommand{\dd}{\mathrm{d}}
\newcommand{\diver}{\mathrm{div}}
\newcommand{\curl}{\mathrm{curl}}
\newcommand{\id}{\mathrm{Id}}
\newtheorem{lemma}{Lemma}
\newtheorem{theorem}{Theorem}

\title{An iterative algorithm for sparse and constrained recovery with applications to divergence-free current reconstructions in magneto-encephalography}
\author{Ignace Loris and Caroline Verhoeven}

\begin{document}
\maketitle
\begin{abstract}
We propose an iterative algorithm for the minimization of a $\ell_1$-norm penalized least squares functional, under additional linear constraints. The algorithm is fully explicit: it uses only matrix multiplications with the three matrices present in the problem (in the linear constraint, in the data misfit part and in penalty term of the functional). None of the three matrices must be invertible. Convergence is proven in a finite-dimensional setting.
We apply the algorithm to a synthetic problem in magneto-encephalography where it is used for the reconstruction of divergence-free current densities subject to a sparsity promoting penalty on the wavelet coefficients of the current densities. We discuss the effects of imposing zero divergence and of imposing joint sparsity (of the vector components of the current density) on the current density reconstruction.
%\keywords{inverse problems\and sparsity\and convex optimization\and iterative algorithm\and wavelets\and magneto-encephalography}
\end{abstract}

\section{Introduction}

In magneto-encephalography (MEG) an image of an electrical current density is reconstructed from measurements of the magnetic field outside the scalp. The magnetic field generated by these currents is very weak compared to the environment; special precautions are taken to minimize these external effects on the observed data. Another characteristic of MEG imaging is the very low number of data: typically only a few hundred of measurements are taken. When using a current density representation of reasonable size (spatial resolution) these data are not complete. This implies having to solve an underdetermined system of equations. Extra conditions need to be imposed to define a unique current density reconstruction.

In \cite{FoP:2008} the use of a sparsity promoting penalty, together with an efficient representation of the current density in terms of wavelets \cite{Daubechies1992,Mallat2009} was proposed. In other words, the assumption that the unknown current density can be represented with a small number of non-zero wavelet coefficients, was used as a priori information to regularize the inversion.

The regularization of the MEG inverse problem by a sparsity assumption was carried out in practice in \cite{FoP:2008} by adding an $\ell_1$-norm penalty term to a quadratic cost function for the data misfit. The $\ell_1$-norm is a popular sparsity promoting penalty \cite{candes08:introCS,Bruckstein.Donoho.ea2009} as it allows for convex optimization techniques to be used instead of algorithms with combinatorial complexity. It was shown in \cite{Daubechies2004b} that such a penalty regularizes the linear inverse problem, and convergence of an iterative soft-thresholding algorithm for the minimization of an $\ell_1$-norm penalized least squares functional was proven in a Hilbert space setting. In the present work we extend the method of \cite{FoP:2008,Daubechies2004b} to incorporate linear constraints (e.g. to impose zero divergence on the reconstructed current densities) and to handle a more general sparsity promoting penalty.

In the first, mathematical, part of this paper we propose a new iterative algorithm for the minimization of an $\ell_1$-norm penalized least squares functional, under additional linear constraints:
\begin{equation}
\hat x=\arg\min_{Bx=b}\|Kx-y\|^2+2\lambda \|Ax\|_1.
\label{functional}
\end{equation}
Here $K$ is the matrix that defines the linear relation between the unknown model $x$ and the data $y$; $Bx=b$ is a linear constraint on the solution. $A$ is a matrix mixing the variables in the non-smooth penalty term $\|Ax\|_1$. It need not be invertible in our approach (e.g. $A=\mathrm{grad}$ would correspond to a total variation penalty \cite{Rudin.Osher.ea1992}). We write the variational equations corresponding to this problem and derive a simple iterative algorithm. This algorithm consists of a single loop and each step in the loop is given explicitly in terms of matrix multiplications by $K$, $A$ and $B$ (and their transposes). We prove the convergence of this algorithm for general $K$, $A$ and $B$ (subject to a bound on their norms) in a finite-dimensional setting. In problem (\ref{functional}) the $\ell_1$-norm penalty may be replaced by another convex function $H(Ax)$. The proposed algorithm can be modified to apply to this case as well, as long the proximity operator of $H$ is known.

The proposed algorithm reduces to the generalized iterative soft-thresholding algorithm of \cite{Loris.Verhoeven2011} when the linear constraints $Bx=b$ are removed. We will indicate below in which special cases (e.g. $A=\id$) our algorithm is also derivable by the method in \cite{Zhang.Burger.ea2011}. In these special cases we also comment on the conditions on the matrices that are necessary for guaranteeing convergence; in particular, we indicate where our conditions on the matrices $K$, $A$ or $B$ are less strict than the ones derivable from the work in \cite{Zhang.Burger.ea2011}.

In the second part of this paper, we apply the algorithm to an inverse problem loosely based on magneto-encephalography. We shall assume a linear relationship between an unknown current density $\vec{J}$ and a measured magnetic field $\vec{B}$ in the form of the Biot-Savart law. Furthermore we assume that the data are contaminated by Gaussian noise. As in \cite{FoP:2008} we will impose sparsity on the wavelet expansion of the current density $\vec{J}$. That is, we will use the proposed iterative algorithm to solve problem (\ref{functional}) where $x$ is the current density $\vec{J}$, $K$ a matrix corresponding to a discretized Biot-Savart law, $A$ the wavelet transform, and $Bx=b$ represents the linear constraints $\diver(\vec{J})=0$. In this last point lies the main difference with the simulations in \cite{FoP:2008}: we shall incorporate into the reconstruction procedure the assumption that the current density is divergence-free; this was not done in \cite{FoP:2008}. Our approach does not use divergence-free wavelets \cite{Lemarie-Rieusset1991,MR1191345,STEVENSON2011}, but relies on an explicit linear constraint for finding a divergence-free reconstruction.

On a synthetic problem, we investigate the effect of the divergence-free nature of the current distribution on the sparse reconstruction. That is, we compare reconstructions from the same measurement data with and without the constraint. Secondly, we investigate the effect on the reconstruction of imposing a ``joint sparsity'' \cite{Fornasier.Rauhut2008a} condition on the two vector components of the current density. This means that at each position both vector components are simultaneously zero or simultaneously non-zero. Joint sparsity in MEG was first discussed in \cite{FoP:2008} as a way of improving reconstruction quality. The proposed iterative algorithm can also handle penalties that promote joint sparsity (see discussion at the end of Section~\ref{proofsection}).

Besides MEG, other applications of $\ell_1$-penalized least squares under linear constraints exist. One is found in the portfolio selection problem described in \cite{Brodie.Daubechies.ea2009}. Such problems are often of a smaller size (fewer variables) and can sometimes also be solved via a non-iterative procedure. Another application of the minimization problem (\ref{functional}) is found in the formulation of a modified Total Variation model in the context of image processing tasks  \cite{Litvinov2011}.

In this paper we will make frequent use of the (non-linear) soft-thresholding operator which is defined by:
\begin{equation}
S_\lambda(z)=
\left\{
\begin{array}{ll}
z-\frac{z}{|z|}\lambda &  |z|>\lambda\\
0 & |z|\leq \lambda
\end{array}
\right.\label{Sdef}
\end{equation}
and of the projection on the $\ell_\infty$ ball of radius $\lambda$ defined by:
\begin{equation}
P_\lambda(z)=
\left\{
\begin{array}{ll}
\frac{z}{|z|}\lambda &  |z|>\lambda\\
z & |z|\leq \lambda.
\end{array}
\right.\label{Pdef}
\end{equation}
We have that:
\begin{equation}\label{PSproperty}
S_\lambda(z)+P_\lambda(z)=z
\end{equation}
for all $z$. We set  $B^{\infty}_{\lambda}=\{u\ \mathrm{with}\ \|u\|_\infty\leq\lambda\}$ ($\ell_\infty$-ball of radius $\lambda$).

\section{Description of the iterative algorithm}

\label{problemsection}

The variational equations of the minimization problem (\ref{functional}) can be obtained by the introduction of Lagrange multipliers $v$:
\begin{equation}
\min_{x}\|Kx-y\|^2+2\lambda \|Ax\|_1-2\langle v,Bx-b\rangle\qquad\mathrm{and}\qquad Bx=b.
\end{equation}
Derivation with respect to $x$ yields:
\begin{displaymath}
K^T(Kx-y)+A^T w-B^Tv=0 \qquad\mathrm{and}\qquad Bx=b,
\end{displaymath}
where $w$ is an element of the subdifferential of
$\lambda\|Ax\|_1$, i.e. $w_i=\lambda\,(Ax)_i/|(Ax)_i|$ if
$(Ax)_i\neq 0$ and $|w_i|\leq\lambda$ if $(Ax)_i=0$. This can be written more compactly as $(Ax)_i=S_\lambda(w_i+(Ax)_i)$ or equivalently
$w_i=P_\lambda(w_i+(Ax)_i)$. We find that the variational equations
corresponding to the problem (\ref{functional}) therefore are:
\begin{equation}
K^T(Kx-y)+A^T w-B^Tv=0,\quad w=\mathbb{P}_\lambda(w+Ax) \quad\mathrm{and}\quad Bx=b,\label{vareq}
\end{equation}
where $\mathbb{P}_{\lambda}(u)$ corresponds to the application of $P_{\lambda}$ (defined in ( formula \ref{Pdef})) on each component of $u$.
We assume that a solution to these equations exists and try to derive an iterative algorithm that converges to such a solution.
By writing $\lambda\|Ax\|_1=\max_{\|w\|_\infty\leq\lambda}\langle w,Ax\rangle$, the minimization
problem (\ref{functional}) can expressed as:
\begin{equation}
\min_{x,Bx=b}\max_{w\in B^{\infty}_{\lambda}} F(x,w,v),
\label{minmax}
\end{equation}
where we have set:
\begin{equation}
F(x,w,v)=\|Kx-y\|^2+2\langle w,Ax\rangle-2\langle v,Bx-b\rangle.
\label{Fdef}
\end{equation}

We write the variational equations (\ref{vareq}) as fixed-point
equations:
\begin{equation}
\left\{
\begin{array}{lcl}
w&=&\mathbb{P}_{\lambda}(w+Ax)\\
x&=&x+K^T(y-Kx)+B^Tv-A^Tw\\
v&=&v-\frac{1}{\alpha}(Bx-b)\\
\end{array}
\right.\label{fixeq}
\end{equation}
and study the predictor-corrector scheme:
\begin{equation}
\left\{
\begin{array}{lcl}
\bar v^{n+1}&=&v^n-(Bx^n-b)\\
\bar x^{n+1}&=&x^n+K^T(y-Kx^n)+B^T\bar v^{n+1}-A^Tw^n\\
w^{n+1}&=&\mathbb{P}_{\lambda}(w^n+A\bar x^{n+1})\\
x^{n+1}&=&x^n+K^T(y-Kx^n)+B^T \bar v^{n+1}-A^Tw^{n+1}\\
v^{n+1}&=&v^n-\frac{1}{\alpha}(Bx^{n+1}-b).
\end{array}
\right.   \label{alg}
\end{equation}
There is a predictor-corrector step on the variables $v$ and $x$ but not on $w$.
Clearly, the fixed-point of this iteration is a solution to the variational
equations (\ref{vareq}). Moreover the algorithm is fully explicit: Each step only requires the application of the matrices $K,A,B$ (and their transposes) and a simple projection $\mathbb{P}_{\lambda}$ (see formula (\ref{Pdef})). There is no non-trivial sub-problem to solve in each step (such as e.g. solving a linear system of equations). In other words, there is no inner loop required for any of the lines in (\ref{alg}).
In the next section we show that, under certain conditions on the operators
$A$, $B$ and $K$, and on the parameter $\alpha$, the proposed algorithm (\ref{alg}) converges to a
solution of the variational equations (\ref{vareq}), and to a
minimizer of the functional (\ref{functional}).

For the special case when $B=0$ and $b=0$ (absence of linear constraints), the algorithm (\ref{alg})
reduces to:
\begin{equation}
\left\{
\begin{array}{lcl}
\bar x^{n+1}&=&x^n+K^T(y-Kx^n)-A^Tw^n\\
w^{n+1}&=&\mathbb{P}_{\lambda}(w^n+A\bar x^{n+1})\\
x^{n+1}&=&x^n+K^T(y-Kx^n)-A^Tw^{n+1}.
\end{array}
\right.\label{gista}
\end{equation}
This algorithm was presented in \cite{Loris.Verhoeven2011} to solve the problem:
\begin{equation}
\hat x=\arg\min_x\|Kx-y\|^2+2\lambda \|Ax\|_1.\label{functional2}
\end{equation}
An important application is the total variation penalty in image analysis ($A=\mathrm{grad}$). A similar algorithm (with predictor-corrector step on the $w$ variable) was proposed independently in \cite{Bonettini2011} for Poisson data. When $A=\id$, and using relation (\ref{PSproperty}), algorithm (\ref{gista}) further simplifies to the traditional iterative soft-thresholding algorithm:
\begin{equation}
x^{n+1}=\mathbb{S}_\lambda\left(x^n+K^T(y-Kx^n)\right),
\end{equation}
where $\mathbb{S}_{\lambda}(u)$ corresponds to the application of $S_{\lambda}$ (defined in formula (\ref{Sdef})) on each component of $u$. This algorithm was discussed in \cite{Daubechies2004b} for solving
\begin{equation}
\min_x\|Kx-y\|^2+2\lambda \|x\|_1.
\end{equation}
An accelerated version of this algorithm, the Fast Iterative Soft-Thresholding Algorithm (FISTA), was derived in \cite{BeT:2009}. Many other algorithms exist as well.

On the other hand, when the constraints $Bx=b$ are maintained, and with $A$ equal to the identity, the algorithm (\ref{alg}) reduces to a constrained version of the iterative soft-thresholding algorithm:
\begin{equation}
\left\{\begin{array}{ll}
\bar v^{n+1}=v^n-(Bx^n-b)\\
x^{n+1}=\mathbb{S}_{\lambda}(x^n+K^T(y-Kx^n)+B^T\bar v^{n+1})\\
v^{n+1}=v^n-\frac{1}{\alpha}(Bx^{n+1}-b),
\end{array}\right.\label{constrl1alg}
\end{equation}
for the problem
\begin{equation}
\hat x=\arg\min_{Bx=b}\|Kx-y\|^2+2\lambda \|x\|_1.\label{constrproblem}
\end{equation}
We will use algorithm (\ref{constrl1alg}) for an application in magneto-encephalography in section \ref{MEGsection}.  Although it was not included in \cite{Zhang.Burger.ea2011}, algorithm (\ref{constrl1alg}) for problem (\ref{constrproblem}), could also have been derived from the Bregman framework of \cite{Zhang.Burger.ea2011}. At the end of section \ref{proofsection} we will comment on the difference between conditions of convergence that the matrices $K$ and $B$ have to satisfy to guarantee convergence of (\ref{constrl1alg}) in our approach and in \cite{Zhang.Burger.ea2011}.

Finally, by setting $K=0$ and $A=\id$ in problem (\ref{functional}) one recovers the so-called $\ell_1$ basis pursuit problem \cite{Chen.Donoho.ea1998}:
\begin{equation}
\arg\min_{Bx=b}\|x\|_1\label{BPproblem}
\end{equation}
for which the algorithms (\ref{alg}) and (\ref{constrl1alg}) reduce to:
\begin{equation}
\left\{\begin{array}{ll}
\bar v^{n+1}=v^n-(Bx^n-b)\\
x^{n+1}=\mathbb{S}_{\lambda}(x^n+B^T\bar v^{n+1})\\
v^{n+1}=v^n-\frac{1}{\alpha}(Bx^{n+1}-b),
\end{array}\right.\label{BPalg}
\end{equation}
This algorithm was discussed in \cite{Zhang.Burger.ea2011} (using different notation and auxiliary variables) in a Bregman framework.

Taking these special cases into account we can say that the proposed algorithm (\ref{alg}) combines the generalized iterative soft-thresholding algorithm (\ref{gista}) of \cite{Loris.Verhoeven2011} with the basis pursuit algorithm (\ref{BPalg}) into a single unified algorithm.

\section{Proof of convergence}
\label{proofsection}

In this section, we prove the convergence of algorithm (\ref{alg}) and show that this yields a minimum of functional (\ref{functional}).

\begin{lemma} If $u^+=\mathbb{P}_{\lambda}(u^-+\Delta)$, with $\mathbb{P}_{\lambda}$ the projection on the convex set $B_\lambda^\infty$, then
\begin{equation}
\|u-u^+\|^2\leq \|u-u^-\|^2-\|u^--u^+\|^2-2\langle u-u^+,\Delta\rangle
\label{lemmaueq}
\end{equation}
for all $u\in B_\lambda^\infty$. \label{lemmau}
\end{lemma}
\begin{proof}
As $\mathbb{P}_{\lambda}$ is the projection on a non-empty closed convex set one has:
\begin{displaymath}
\langle u-\mathbb{P}_\lambda(u'),u'-\mathbb{P}_\lambda(u')\rangle\leq 0
\end{displaymath}
for all $u\in B_\lambda^\infty$ and all $u'$. Choosing $u'=u^-+\Delta$ and $\mathbb{P}_{\lambda}(u')=u^+$ yields
\begin{displaymath}
\langle u-u^+,u^-+\Delta-u^+\rangle\leq 0.
\end{displaymath}
Replacing $\langle u-u^+,u^--u^+\rangle$ by $\left(\|u-u^+\|^2+\|u^--u^+\|^2-\|u-u^-\|^2\right)/2$
yields
\begin{displaymath}
\|u-u^+\|^2+\|u^--u^+\|^2-\|u-u^-\|^2+2\langle u-u^+,\Delta\rangle\leq 0
\end{displaymath}
which is the desired result.
\end{proof}

The operator $\mathbb{P}_{\lambda}$ and the convex set $B_\lambda^\infty$ in Lemma \ref{lemmau} may be replaced with a projection on any non-empty closed convex set. In particular, when $u^+=u^-+\Delta$ (i.e. $B_\lambda^\infty$ replaced by the whole space and $\mathbb{P}_{\lambda}$ replaced by the identity), one has that:
\begin{equation}
\|u-u^+\|^2=\|u-u^-\|^2-\|u^--u^+\|^2-2\langle u-u^+,\Delta\rangle\label{temp2}
\end{equation}
for all $u$.

\begin{lemma} If $(x^{n+1},w^{n+1},v^{n+1})$ and $(x^n,w^n,v^n)$ are related by iteration (\ref{alg}) then
\begin{equation}
\begin{array}{l}
\|x-x^{n+1}\|^2+\|w-w^{n+1}\|^2+\alpha\|v-v^{n+1}\|^2\leq \\
\qquad\qquad \|x-x^n\|^2+\|w-w^n\|^2+
\alpha\|v-v^n\|^2-\|x^n-x^{n+1}\|^2\\
\qquad\qquad-\|w^n-w^{n+1}\|^2
-\alpha\|v^n-v^{n+1}\|^2-\|K(x-x^n)\|^2\\
\qquad\qquad +\|K(x^n-x^{n+1})\|^2-\|B(x-x^n)\|^2
+\|B(x-x^{n+1})\|^2\\
\qquad\qquad +\|B(x^n-x^{n+1})\|^2
-\|A^T(w-w^n)\|^2\\
\qquad\qquad +\|A^T(w-w^{n+1})\|^2
+\|A^T(w^{n+1}-w^n)\|^2\\
\qquad\qquad  +F(x,w^{n+1},v^{n+1})-F(x^{n+1},w,v)\\
\qquad\qquad +2\frac{\alpha-1}{\alpha}\langle B(x-x^{n+1}),Bx^{n+1}-b\rangle
\end{array}
\label{ineqlemma3}
\end{equation}
for all $x,v$ and all $w\in B^{\infty}_{\lambda}$.
\label{lemma3}
\end{lemma}
\begin{proof}
From Lemma \ref{lemmau} and equation (\ref{temp2}), we find:
\begin{displaymath}
\begin{array}{l}
\|x-x^{n+1}\|^2+\|w-w^{n+1}\|^2+\alpha\|v-v^{n+1}\|^2\leq \|x-x^n\|^2+\|w-w^n\|^2\\
\qquad +\alpha\|v-v^n\|^2-\|x^n-x^{n+1}\|^2-\|w^n-w^{n+1}\|^2
-\alpha\|v^n-v^{n+1}\|^2\\
\qquad -2\langle x-x^{n+1},K^T(y-Kx^n)\rangle
-2\langle x-x^{n+1},B^T\bar v^{n+1}-A^Tw^{n+1}\rangle\\
\qquad -2\langle w-w^{n+1},A\bar x^{n+1}\rangle
+2\langle v-v^{n+1},Bx^{n+1}-b\rangle.
\end{array}
\end{displaymath}
As (\ref{alg}) implies that $\bar v^{n+1}=v^{n+1}+B(x^{n+1}-x^n)+\frac{1-\alpha}{\alpha}(Bx^{n+1}-b)$
and that $\bar x^{n+1}=x^{n+1}+A^T(w^{n+1}-w^n)$,
this can be written as:
\begin{displaymath}
\begin{array}{l}
\|x-x^{n+1}\|^2+\|w-w^{n+1}\|^2+\alpha\|v-v^{n+1}\|^2\leq \|x-x^n\|^2+\|w-w^n\|^2\\
\quad+
\alpha\|v-v^n\|^2-\|x^n-x^{n+1}\|^2-\|w^n-w^{n+1}\|^2
-\alpha\|v^n-v^{n+1}\|^2\\
\quad -2\langle K(x-x^{n+1}),y-Kx^n\rangle+2\langle A(x-x^{n+1}),w^{n+1}\rangle\\
\quad
-2\langle B(x-x^{n+1}),v^{n+1}\rangle
-2\langle B(x-x^{n+1}),B(x^{n+1}-x^n)\rangle\\
\quad+2\frac{\alpha-1}{\alpha}\langle B(x-x^{n+1}),Bx^{n+1}-b\rangle
-2\langle w-w^{n+1},Ax^{n+1}\rangle\\
\quad-2\langle A^T(w-w^{n+1}),A^T(w^{n+1}-w^n)\rangle
+2\langle v-v^{n+1},Bx^{n+1}-b\rangle.
\end{array}
\end{displaymath}
The terms in $\langle v^{n+1},Bx^{n+1}\rangle$ and $\langle w^{n+1},Ax^{n+1}\rangle$  drop and the remaining terms can be re-arranged to yield:
\begin{displaymath}
\begin{array}{l}
\|x-x^{n+1}\|^2+\|w-w^{n+1}\|^2+\alpha\|v-v^{n+1}\|^2\leq \|x-x^n\|^2+\|w-w^n\|^2\\
\quad +\alpha\|v-v^n\|^2-\|x^n-x^{n+1}\|^2-\|w^n-w^{n+1}\|^2
-\alpha\|v^n-v^{n+1}\|^2\\
\quad -2\langle K(x-x^{n+1}),y-Kx^n\rangle-2\langle A^T(w-w^{n+1}),A^T(w^{n+1}-w^n)\rangle\\
\quad -2\langle B(x-x^{n+1}),B(x^{n+1}-x^n)\rangle+2\frac{\alpha-1}{\alpha}\langle B(x-x^{n+1}),Bx^{n+1}-b\rangle\\
\quad +2\langle w^{n+1},Ax\rangle-2\langle w,Ax^{n+1}\rangle
-2\langle v^{n+1},Bx-b\rangle+2\langle v, Bx^{n+1}-b\rangle.
\end{array}
\end{displaymath}
By rewriting the following inner products:
\begin{displaymath}
\begin{array}{l}
-2\langle K(x-x^{n+1}),y-Kx^n\rangle =\|Kx-y\|^2-\|Kx^{n+1}-y\|^2\\
 \qquad\qquad\qquad\qquad\qquad\qquad -\|K(x-x^n)\|^2+\|K(x^n-x^{n+1})\|^2\\[2mm]
-2\langle A^T(w-w^{n+1}),A^T(w^{n+1}-w^n) \rangle=\|A^T(w-w^{n+1})\|^2\\
 \qquad\qquad\qquad\qquad\qquad\qquad +\|A^T(w^{n+1}-w^n)\|^2  -\|A^T(w-w^n)\|^2\\[2mm]
-2\langle B(x-x^{n+1}),B(x^{n+1}-x^n) \rangle=\|B(x-x^{n+1})\|^2\\
 \qquad\qquad\qquad\qquad\qquad\qquad +\|B(x^{n+1}-x^n)\|^2 -\|B(x-x^n)\|^2,
\end{array}
\end{displaymath}
and by using the expression (\ref{Fdef}) of $F(x,v,w)$ in:
\begin{displaymath}
\begin{split}
2\langle w^{n+1},Ax\rangle&-2\langle w,Ax^{n+1}\rangle
-2\langle v^{n+1},Bx-b\rangle+2\langle v, Bx^{n+1}-b\rangle=\\
&F(x,w^{n+1},v^{n+1})-F(x^{n+1},w,v)-
\|Kx-y\|^2+\|Kx^{n+1}-y\|^2,
\end{split}
\end{displaymath}
the previous inequality can be written as (\ref{ineqlemma3}), which proves the lemma.
\end{proof}

\begin{lemma}
If $(\hat x,\hat w,\hat v)$ satisfies the variational equations
(\ref{vareq}), then
\begin{equation}
F(\hat x,w,v)-F(x,\hat w,\hat v)\leq -\|K(x-\hat x)\|^2
\label{tmp0}
\end{equation}
for all $x,v$ and all $w\in B^{\infty}_{\lambda}$.
\label{lemmagap}
\end{lemma}
\begin{proof}
To prove inequality (\ref{tmp0}), we use
Lemma \ref{lemma3}, where we replace both $(x^n,w^n,v^n)$ and $(x^{n+1},w^{n+1},v^{n+1})$ by
$(\hat x,\hat w,\hat v)$; this is allowed because $(\hat x,\hat w,\hat v)$ satisfies (\ref{fixeq}) which are the fixed-point equations of algorithm (\ref{alg}).
Then relation (\ref{ineqlemma3}) becomes:
\begin{equation}
\begin{split}
\|x-\hat x\|^2+\|w-\hat w\|^2+\alpha\|v-\hat v\|^2\leq & \|x-\hat x\|^2+\|w-\hat w\|^2+\alpha\|v-\hat v\|^2\\
&-\|K(x-\hat x)\|^2\\
&-\|B(x-\hat x)\|^2+\|B(x-\hat x)\|^2\\
&-\|A^T(w-\hat w)\|^2+\|A^T(w-\hat w)\|^2\\
&+F(x,\hat w,\hat v)-F(\hat x,w,v),
\end{split}
\end{equation}
for all $x,v$ and all $w\in B^{\infty}_{\lambda}$. This implies inequality (\ref{tmp0}).
\end{proof}

We now show that a solution of equations (\ref{vareq}) solves
the minimization problem (\ref{functional}).
\begin{theorem}
\label{theorem1}
If $(\hat x,\hat w,\hat v)$ satisfies the variational equations (\ref{vareq}) then $\hat x$ is a solution of the minimization problem (\ref{functional}).
\end{theorem}
\begin{proof}
If $(\hat x,\hat w,\hat v)$ is a solution of (\ref{vareq}) it follows from
Lemma \ref{lemmagap} that $F(\hat x,w,v)\leq F(x,\hat w,\hat v)$ for all
$x,v$ and all $w\in B^{\infty}_{\lambda}$, which means:
\begin{displaymath}
\|K\hat x-y\|^2+2\langle w,A\hat x\rangle\leq
\|Kx-y\|^2+2\langle \hat w,Ax\rangle -2\langle \hat v,Bx-b\rangle.
\end{displaymath}
Taking the maximum over $w\in B^{\infty}_{\lambda}$ in the left hand side gives
\begin{displaymath}
\|K\hat x-y\|^2+2\lambda\|A\hat x\|_1
\leq \|Kx-y\|^2+2\langle \hat w,Ax\rangle-2\langle \hat v,Bx-b\rangle
\end{displaymath}
and since $\langle\hat w,Ax\rangle\leq
\max_{\|\tilde w\|_\infty\leq\lambda}\langle\tilde w,Ax\rangle
=\|Ax\|_1$ one finds:
\begin{displaymath}
\|K\hat x-y\|^2+2\lambda\|A\hat x\|_1
\leq \|Kx-y\|^2+2\lambda\|Ax\|_1-2\langle \hat v,Bx-b\rangle.
\end{displaymath}
for all $x,v$. As we minimize under the condition that $Bx=b$, we have that
$\langle \hat v,Bx-b\rangle=0$ and find
\begin{displaymath}
\|K\hat x-y\|^2+2\lambda\|A\hat x\|_1\leq \|Kx-y\|^2+2\lambda\|Ax\|_1,
\end{displaymath}
for all $x$ for which $Bx=b$. As $B\hat x=b$ this proves the theorem.
\end{proof}

\begin{theorem}
If the set $\{x,\ \mathrm{with}\ Bx=b\}$ is non-empty, $\|AA^T\|<1$,
$\|\frac{1}{2}K^TK+B^TB\|<1$ and $\alpha>\frac{1}{2}$,
then the sequence $(x^n,w^n,v^n)_{n\in\mathbb{N}}$ defined by the iteration
\begin{equation}
\left\{
\begin{array}{lcl}
\bar v^{n+1}&=&v^n-(Bx^n-b)\\
\bar x^{n+1}&=&x^n+K^T(y-Kx^n)+B^T\bar v^{n+1}-A^Tw^n\\
w^{n+1}&=&\mathbb{P}_{\lambda}(w^n+A\bar x^{n+1})\\
x^{n+1}&=&x^n+K^T(y-Kx^n)+B^T \bar v^{n+1}-A^Tw^{n+1}\\
v^{n+1}&=&v^n-\frac{1}{\alpha}(Bx^{n+1}-b)
\end{array}\right.
\label{theoalg}
\end{equation}
converges to a solution $(x^\dagger,w^\dagger,v^\dagger)$ of the
variational equations (\ref{vareq}), and a solution of the minimization problem (\ref{functional}).
\label{theorem2}
\end{theorem}
\begin{proof}
If $\{x,\ \mathrm{with}\ Bx=b\}$ is not empty, there is a solution to (\ref{functional}),
implying that there exists a solution $(\hat x,\hat w,\hat v)$
to the variational equations (\ref{fixeq}). We now use Lemma \ref{lemma3}
with $(x,w,v)=(\hat x,\hat w,\hat v)$ and find:
\begin{displaymath}
\begin{array}{l}
\|\hat x-x^{n+1}\|^2+\|\hat w-w^{n+1}\|^2+\alpha\|\hat v-v^{n+1}\|^2
\leq\|\hat x-x^{n}\|^2+\|\hat w-w^{n}\|^2\\
\quad +\alpha\|\hat v-v^{n}\|^2-\|x^n-x^{n+1}\|^2-\|w^n-w^{n+1}\|^2
-\alpha\|v^{n+1}-v^n\|^2\\
\quad-\|K(\hat x-x^{n})\|^2+\|K(x^n-x^{n+1})\|^2-\|B(\hat x-x^{n})\|^2
+\|B(\hat x-x^{n+1})\|^2\\
\quad+\|B(x^n-x^{n+1})\|^2
-\|A^T(\hat w-w^n)\|^2+\|A^T(\hat w-w^{n+1})\|^2
\\
\quad+\|A^T(w^{n+1}-w^n)\|^2-\|K(\hat x-x^{n+1})\|^2+2\frac{1-\alpha}{\alpha}\langle B(x^{n+1}-\hat x),Bx^{n+1}-b\rangle,
\end{array}
\end{displaymath}
where we also used relation (\ref{tmp0}) with $(x,w,v)=(x^{n+1},w^{n+1},v^{n+1})$.
As $B\hat x=b$ we now use that:
\begin{displaymath}
2\frac{1-\alpha}{\alpha}\langle B(x^{n+1}-\hat x),Bx^{n+1}-b\rangle=
2\frac{1-\alpha}{\alpha}\|Bx^{n+1}-b\|^2=2\alpha(1-\alpha)\|v^{n+1}-v^n\|^2,
\end{displaymath}
and:
\begin{displaymath}
\begin{array}{lcl}
-\|K(\hat x-x^n)\|^2-\|K(x^{n+1}-\hat x)\|^2 &=&-\frac{1}{2}\|K(x^{n+1}-x^n)\|^2\\
&&\qquad\qquad\qquad -\frac{1}{2}\|K(2\hat x-x^{n+1}-x^n)\|^2\\[2mm]
&\leq& -\frac{1}{2}\|K(x^{n+1}-x^n)\|^2,
\end{array}
\end{displaymath}
and reorder to obtain:
\begin{displaymath}
\begin{array}{l}
\|\hat x-x^{n+1}\|^2-\|B(\hat x-x^{n+1})\|^2+\|\hat w-w^{n+1}\|^2
-\|A^T(\hat w-w^{n+1})\|^2\\
\qquad +\alpha\|\hat v-v^{n+1}\|^2\leq  \|\hat x-x^n\|^2-\|B(\hat x-x^n)\|^2+\|\hat w-w^n\|^2\\
\qquad\qquad -\|A^T(\hat w-w^n)\|^2+\alpha\|\hat v-v^n\|^2\\
\qquad\qquad -\Big(\|x^{n+1}-x^n\|^2-\frac{1}{2}\|K(x^{n+1}-x^n)\|^2-\|B(x^n-x^{n+1})\|^2\Big)\\
\qquad\qquad -\Big(\|w^{n+1}-w^n\|^2-\|A^T(w^{n+1}-w^n)\|^2\Big)\\
\qquad\qquad -\alpha(2\alpha-1)\|v^{n+1}-v^n\|^2.
\end{array}
\end{displaymath}
As we assume that $\|AA^T\|<1$ and $\|\frac{1}{2}K^TK+B^TB\|<1$ we can introduce regular
square matrices $L$, $U$ and $V$ by $L^TL=1-\frac{1}{2}K^TK-B^TB$, $U^TU=1-B^TB$ and $V^TV=1-AA^T$ to find:
\begin{equation}
\begin{split}
\|U(\hat x-x^{n+1})\|^2&+\|V(\hat w-w^{n+1})\|^2+\alpha\|\hat v-v^{n+1}\|^2\\
\leq&\|U(\hat x-x^{n})\|^2+\|V(\hat w-w^{n})\|^2+\alpha\|\hat v-v^n\|^2\\
& -\|L(x^n-x^{n+1})\|^2-\|V(w^n-w^{n+1})\|^2\\
&-\alpha(2\alpha-1)\|v^{n+1}-v^n\|^2.
\end{split}
\end{equation}
Summing from $M$ to $N>M$ one finds:
\begin{equation}
\begin{split}
\|U(\hat x-x^{N+1})\|^2&+\|V(\hat w-w^{N+1})\|^2+\alpha\|\hat v-v^{N+1}\|^2\\
\leq&\|U(\hat x-x^M)\|^2+\|V(\hat w-w^M)\|^2+\alpha\|\hat v-v^M\|^2\\
& -\sum_{n=M}^N\Big(\|L(x^n-x^{n+1})\|^2+\|V(w^n-w^{n+1})\|^2 \\
&\qquad\qquad+\alpha(2\alpha-1)\|v^{n+1}-v^n\|^2\Big).
\label{temp1}
\end{split}
\end{equation}
Since $\alpha>\frac{1}{2}$ the summation on the right hand side is negative.
As $U$ and $V$ are invertible, it follows that the sequence $(x^n,w^n,v^n)$
is bounded. And, as we work in a finite dimensional space, there is a convergent subsequence
$(x_{n_j},w_{n_j},v_{n_j})\stackrel{j\rightarrow\infty}{\rightarrow}
(x^\dagger,w^\dagger,v^\dagger)$. It also follows
from inequality (\ref{temp1}) that:
\begin{displaymath}
\begin{split}
\sum_{n=M}^N&\Big(\|L(x^n-x^{n+1})\|^2+\|V(w^n-w^{n+1})\|^2
+\alpha(2\alpha-1)\|v^{n+1}-v^n\|^2\Big)\\
&\leq\|U(\hat x-x^M)\|^2
+\|V(\hat w-w^M)\|^2+\alpha\|\hat v-v^M\|^2
\end{split}
\end{displaymath}
As $\alpha>\frac{1}{2}$, $\|L(x^n-x^{n+1})\|^2$, $\|V(w^n-w^{n+1})\|^2$ and $\|v^{n+1}-v^n\|^2$
tend to zero for large $n$, which implies that $\|x^n-x^{n+1}\|^2$ and
$\|w^n-w^{n+1}\|^2$ tend to zero as well. It follows that the subsequence
$(x_{n_j+1},w_{n_j+1},v_{n_j+1})$ also converges to
$(x^\dagger,w^\dagger,v^\dagger)$ and that $(x^\dagger,w^\dagger,v^\dagger)$
satisfies the fixed-point equations (\ref{fixeq}). We can
therefore choose $(\hat x, \hat w, \hat v)=(x^\dagger,w^\dagger,v^\dagger)$ in
relation (\ref{temp1}) to find:
\begin{displaymath}
\begin{array}{l}
\|U(x^\dagger-x^{N+1})\|^2+\|V(w^\dagger-w^{N+1})\|^2+\alpha\|v^\dagger-v^{N+1}\|^2\\
\qquad\qquad\qquad\leq\|U(x^\dagger-x^M)\|^2
+\|V(w^\dagger-w^M)\|^2
+\alpha\|v^\dagger-v^M\|^2
\end{array}
\end{displaymath}
for all $N>M$. As there is a convergent subsequence of
$(x^n,w^n,v^n)$, the right hand side of this expression can
be made arbitrarily small by taking $M=n_j$ large enough.
Hence the left hand side will be arbitrarily small for all
$N$ larger than this $M$. This proves convergence of the
whole sequence $(x^n,w^n,v^n)$ to $(x^\dagger,w^\dagger,v^\dagger)$.

As $(x^\dagger,w^\dagger,v^\dagger)$ satisfies the fixed-point equations, it follows from Theorem~\ref{theorem1} that $x^\dagger$ is a solution to problem (\ref{functional}).
\end{proof}

\section{Discussion}
\label{discussionsection}
\begin{itemize}

\item
If $\|\frac{1}{2}K^TK+B^TB\|\geq1$ or $\|AA^T\|\geq 1$ one can rescale the matrices and the
variables to arrive at the following iteration:
\begin{equation}
\left\{
\begin{array}{lcl}
\bar v^{n+1}&=&v^n-(Bx^n-b)\\
\bar x^{n+1}&=&x^n+\tau_1 K^T(y-Kx^n)+\tau_3 B^T\bar v^{n+1}-\tau_1 A^Tw^n\\
w^{n+1}&=&\mathbb{P}_{\lambda}(w^n+\frac{\tau_2}{\tau_1}A\bar x^{n+1})\\
x^{n+1}&=&x^n+\tau_1 K^T(y-Kx^n)+\tau_3 B^T \bar v^{n+1}-\tau_1 A^Tw^{n+1}\\
v^{n+1}&=&v^n-\frac{1}{\alpha}(Bx^{n+1}-b)
\end{array}\right.
\label{scaledalg}
\end{equation}
with step size parameters $\tau_1,\tau_2,\tau_3>0$ that satisfy
$\|\tau_1K^TK/2+\tau_3 B^TB\|<1$ and $\tau_2\|AA^T\|<1$.

\item
The $\ell_1$-norm in problem (\ref{functional}) does not necessarily have to be defined as $\|u\|_1=\sum_i|u_i|$.
In section~\ref{MEGsection} we will use an $\ell_1$-norm of the form $\|u\|_1=\sum_i^N \max(|u_{i,1}|,\ldots, |u_{i,m}|)$ for a vector $u\in \mathbb{R}^{N\times m}$. Such a penalty is useful for promoting joint sparsity on the $u_{i,j}$ (for a fixed $i$). Indeed, if e.g. $u_{i,1}$ is non-zero, then all other $u_{i,j}$ ($j\neq 1$) may be as large as $|u_{i,1}|$ as well, without increasing $\max(|u_{i,1}|,\ldots, |u_{i,m}|)$.\\
As $\lambda\max(|z_{1}|,\ldots, |z_{m}|)=\max_{\|w\|_1\leq \lambda}\langle w,z\rangle$, one needs to replace the projection $\mathbb{P}_\lambda$ in (\ref{alg}) by $N$ projections on an $\ell_1$-ball (in $\mathbb{R}^m$) of radius $\lambda$. We denote the projection on an $\ell_1$-ball of radius $\lambda$ in $\mathbb{R}^m$ by $Q_\lambda$. In algorithm (\ref{constrl1alg}), that is used for the special case $A=\id$, one has to replace the component-wise soft-thresholding $S_\lambda$ with a new thresholding function $T_\lambda=\mathrm{Id}-Q_\lambda$ (and apply it to $N$ vectors of size $m$). The operator $T_\lambda$ can be computed as follows \cite{Fornasier.Rauhut2008a}. Let $z\in\mathbb{R}^m$  and order the entries such that $|z_{i_1}|\geq |z_{i_2}|\geq\cdots\geq |z_{i_m}|$. Then:
\begin{equation}
\left\{
\begin{array}{lcl}
\mathrm{for}\ \|z\|_1\leq\lambda &:& T_\lambda(z)=0\\[3mm]
\mathrm{for}\ \|z\|_1> \lambda &:&
\left\{
\begin{array}{lcl} (T_\lambda(z))_{i_j}=\operatorname{sgn}(z_{i_j})(\sum_{k=1}^l|z_{i_k}|-\lambda)/l &&j=1,\dots,l\\[2mm]
(T_\lambda(z))_{i_j}=z_{i_j} && j=l+1,\dots,m
\end{array}
\right.
\end{array}
\right.
\label{Tdef}
\end{equation}
where $l\in\{1,\dots,m\}$ is the largest index satisfying
$|z_{i_l}|\geq(\sum_{k=1}^l|z_{i_k}|-\lambda)/l$.

\item
The $\ell_1$-norm in functional (\ref{functional}) can be replaced by a convex lower semi-continuous function $H$:
\begin{equation}
\hat x=\arg\min_{Bx=b}\|Kx-y\|^2+2H(Ax).
\label{functional3}
\end{equation}
(assuming a minimizer exists).
The projection operator $\mathbb{P}_{\lambda}$ in algorithm (\ref{alg}) then needs to be replaced by the proximity operator of  the convex conjugate of $H$, $H^\ast$, defined by $H^\ast(w)=\sup_x\{\langle w,x\rangle-H(x)\}$ (see e.g. \cite{Combettes.Pesquet2011}):
\begin{equation}
\left\{
\begin{array}{lcl}
\bar v^{n+1}&=&v^n-(Bx^n-b)\\
\bar x^{n+1}&=&x^n+K^T(y-Kx^n)+B^Tv^{n+1}-A^Tw^n\\
w^{n+1}&=&\mathrm{prox}_{H^\ast}(w^n+A\bar x^{n+1})\\
x^{n+1}&=&x^n+K^T(y-Kx^n)+B^T \bar v^{n+1}-A^Tw^{n+1}\\
v^{n+1}&=&v^n-\frac{1}{\alpha}(Bx^{n+1}-b),
\end{array}\right.
\label{genalg}
\end{equation}
which converges under the same conditions as in Theorem~\ref{theorem2} to a minimizer of problem (\ref{functional3}).
The proximity operator of $H$ is defined as $\mathrm{prox}_H(u)=\arg\min_wH(w)+\|w-u\|^2/2$, and $\mathrm{prox}_{H^\ast}=\mathrm{Id}-\mathrm{prox}_{H}$.
It is important to remark that only the proximity operator $\mathrm{prox}_{H}$ of $H$ is needed, not the proximity operator of $H(A\cdot)$.
In fact the convergence of algorithm (\ref{gista}) (i.e. without the linear constraints $Bx=b$) was proven in this more general context in \cite{Loris.Verhoeven2011}.

\item
The functional $\mathcal{F}(x)=\|Kx-y\|^2+2\lambda \|Ax\|_1$, evaluated in the iterates $x^n$, does not decrease monotonically as a function of $n$. Because the iterates $x^n$ do not necessarily satisfy the constraint in every step, it is even possible that $\mathcal{F}(x^n)<\mathcal{F}(\hat x)$ for some $n$. The constraint $Bx=b$ is only satisfied in the limit $n\rightarrow\infty$.

\item
If one wants to solve the $\ell_1$-norm constrained problem
\begin{equation}
\hat x=\arg\min_{Bx=b, \|x\|_1\leq R}\|Kx-y\|^2
\end{equation}
instead of the $\ell_1$-norm penalized problem (\ref{constrproblem}), then one may replace the soft-thresholding $\mathbb{S}_{\lambda}$ in algorithm (\ref{constrl1alg}) by projection on
the $\ell_1$-ball. The algorithm is:
\begin{equation}
\left\{
\begin{array}{lcl}
\bar w^{n+1}&=&w^n-(Bx^n-b)\\
x^{n+1}&=&Q_R\left(x^n+K^T(y-Kx^n)+B^T \bar w^{n+1}\right)\\
w^{n+1}&=&w^n-\frac{1}{\alpha}(Bx^{n+1}-b),
\end{array}
\right.   \label{projalg}
\end{equation}
where $Q_R$ is the projection on the $\ell_1$-ball of radius $R$ (such a projection is explicitly doable by computer; see expression (\ref{Tdef}), with $\lambda$ replaced by $R$, and the paragraph above). This algorithm converges for $\|\frac{1}{2}K^TK+B^TB\|<1$ and $\alpha>1/2$. This can be shown by using Lemma~\ref{lemmau} (for an $\ell_1$-ball instead of an $\ell_\infty$-ball) and proceeding in the same way as in Theorem~\ref{theorem2} (without proof).

\item
In the special case when $A=\id$, the algorithm (\ref{constrl1alg}) could also have been obtained from algorithm ($A_0$), formula (3.2) of \cite{Zhang.Burger.ea2011} (it wasn't done explicitly). This can be achieved by the choice $H\rightarrow\frac{1}{2}\|Kx-y\|^2$, $A\rightarrow B$, $J\rightarrow\|\cdot\|_1$ and $Q_0\rightarrow\id-K^TK-B^TB$ and $C\rightarrow\alpha$ in \cite{Zhang.Burger.ea2011}.
%This can be done by taking $H=\frac{1}{2}\|Kx-y\|^2$, $J=\|\cdot\|_1$ and $Q_0=1-K^TK-B^TB$ and $C=\alpha$ (the matrix $B$ is denoted by $A$ in formula (3.2)) in \cite{Zhang.Burger.ea2011}.
However, the assumption in \cite{Zhang.Burger.ea2011} that $Q_0$ is positive definite
(necessary to prove the convergence of the algorithm),
amounts to having $\|K^TK+B^TB\|<1$. We
have shown here that the condition $\|\frac{1}{2}K^TK+B^TB\|<1$ is
already sufficient to guarantee convergence.

As already mentioned, when $A=\id$ and $K=0$ the proposed algorithm (\ref{alg}) reduces to the algorithm (\ref{BPalg}) for the problem (\ref{BPproblem}). This algorithm was also (re)derived in \cite{Zhang.Burger.ea2011} under a slightly different form (see equation (5.6) of \cite{Zhang.Burger.ea2011}).
\end{itemize}

\section{Application to magneto-encephalography}
\label{MEGsection}

The goal of magneto-encephalography (MEG) is to determine a current
density $\vec{J}$ in the brain by measuring (a component of) the
magnetic field $\vec{B}$ induced by $\vec{J}$, in several points
outside the scalp. We assume that $\vec{B}$ and $\vec{J}$ are linked by the Biot-Savart law:
\begin{equation}
\vec{B}(\vec{r})=\frac{\mu_0}{4\pi}\int_V\vec{J}(\vec{r}\,')\times
\frac{\vec{r}-\vec{r}\,'}{\left|\vec{r}-\vec{r}\,'\right|^3}\,\dd V',
\label{equ:bios}
\end{equation}
with $\mu_0=4\pi\times 10^{-7}\mathrm{Vs}/\mathrm{Am}$ and $V$ the
volume in which the current flows. The conservation of charges
implies that $\diver(\vec{J})=0$.
For more information, see \cite{DPT:2001} and references therein.

In this section, we pose and solve a synthetic inverse problem inspired
by this problem. We consider a thin spherical shell $V$ centered at
the origin, with an outer radius of 9cm and  a thickness of 1mm.
Measurements are made at 500 random points $\vec{r}_i$ ($i:1\ldots500$) uniformly distributed
on the upper hemisphere with a distance of 10cm to the origin. The data
is composed of the radial component of the magnetic field
$B_r(\vec{r})$ measured in these $500$ points. For simplicity all the current densities considered below do not depend on the radius and do not have a radial component. In other words, we treat a 2D problem.

In order to discretize the problem, we use the ``cubed sphere'' parametrization introduced in \cite{RIP:1996}; it maps the sphere to the six sides of a cube and a regular grid is then used
on each of these six sides. We choose a grid with $64^2$ voxels on each side.

As input model (that we will want to reconstruct) we choose the current density distribution $\vec{J}_\mathrm{in}$ shown in Figure~\ref{fig:MEGin}. This model is divergence-free by construction.
The matrix $K$ encodes the Biot-Savart law for the radial component $B_r(\vec{r}_i)=\vec{B}\cdot\vec{e}_r(\vec{r}_i)$ of $\vec{B}$ in the $500$ measurement points $\vec{r}_i$ ($i:1\ldots500$). As
\begin{equation}
\begin{array}{lcl}
\displaystyle \vec{B}(\vec{r}_i)\cdot\vec{e}_r(\vec{r}_i)&=&\displaystyle\frac{\mu_0}{4\pi}\int_V
\left(\vec{J}(\vec{r}')\times \frac{\vec{r}_i-\vec{r}'}{|\vec{r}_i-\vec{r}'|^3}\right)\cdot\vec{e}_r(\vec{r}_i)\,\mathrm{d}V'\\[4mm]
&=& \displaystyle\frac{\mu_0}{4\pi}\int_V \left(\frac{\vec{r}_i-\vec{r}'}{|\vec{r}_i-\vec{r}'|^3}\times \vec{e}_r(\vec{r}_i)\right)
\cdot\vec{J}(\vec{r}')\,\mathrm{d}V',
\end{array}
\end{equation}
we set:
\begin{equation}
(K\vec{J})_i=\!\!\!\!\sum_{\mathrm{all\ voxels}}K_{i,\mathrm{voxel}}\cdot
\vec{J}_{\mathrm{voxel}},\quad
K_{i,\mathrm{voxel}}=\frac{\mu_0}{4\pi}\int_{\mathrm{voxel}}
\frac{\vec{r}_i-\vec{r}\,'}{|\vec{r}_i-\vec{r}\,'|^3}\times\vec{e}_r(\vec{r}_i)\,\dd V'
\end{equation}
The data $y$ is obtained from the synthetic input model $\vec{J}_{\mathrm{in}}$. More precisely, the data is constructed by setting $y=K\vec{J}_{\mathrm{in}}+\epsilon$, where $\epsilon$ is Gaussian noise. We choose a noise level of $10\%$: $\|\epsilon\|=0.1\times \|K\vec{J}_{\mathrm{in}}\|$.

\begin{figure}
\centering\includegraphics[width=12cm]{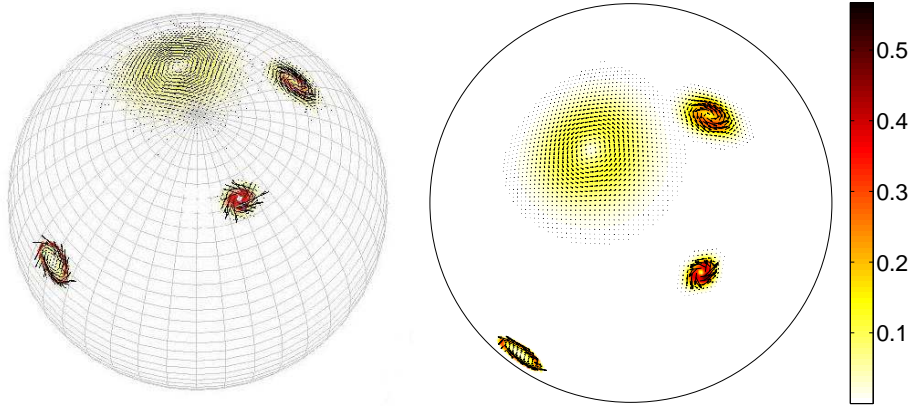}
\caption{Input model $\vec{J}_{\mathrm{in}}$ on the sphere (left) and its top view (right). The background color map is proportional to the local norm of the current density vector $\vec{J}_{\mathrm{in}}$.}
\label{fig:MEGin}
\end{figure}

Our goal now is to reconstruct the current density on each of the voxels of the cubed sphere from the noisy data $y$. The cubed sphere parametrization allows us to use a simple set of wavelet functions on the sphere introduced in \cite{Simons.Loris.ea2011}. They belong to the CDF 4-2 family \cite{CDF:1992}. As the current density has two components, we have to reconstruct $64^2\times6\times2=49152$ coefficients from merely $500$ measurements. The problem is therefore severely under-determined.

To impose the assumed sparsity of $\vec{J}$ in the wavelet basis, we will use an $\ell_1$-penalized least squares functional, with the additional linear constraint $\diver(\vec{J})=0$. We set $\vec{w}$ equal to the list of coefficients of $\vec{J}$ in the wavelet basis. Each element $\vec{w_k}$ of the list $\vec{w}$ (and $\vec{J}$) has two components $w_{k,1},w_{k,2}$ corresponding to the two angular directions on the cubed sphere.
The wavelet transform that maps the wavelet coefficients $\vec{w}$ to the current density $\vec{J}$ is represented by the operator  $W^{-1}$ so that $\vec{J}=W^{-1}\vec{w}$. In fact $W^{-1}$ works on both angular components separately.

The minimization problem we want to solve is:
\begin{equation}
\vec{w}_\mathrm{rec}=
\arg\min_{\vec{w}}\|K W^{-1}\vec{w}-y\|^2+2\lambda H(\vec{w}),
\label{for:megl1}
\end{equation}
where the reconstructed current density is $\vec{J}_\mathrm{rec}=W^{-1}\vec{w}_\mathrm{rec}$.
We consider $4$ distinct cases:
\begin{enumerate}
\def\theenumi{\alph{enumi}}
\item problem (\ref{for:megl1}) with $H(\vec{w})=\sum_{k}|w_{k,1}|+|w_{k,2}|$ without the constraint $\diver\vec{J}=0$\label{nodiv0nojoint}

\item problem (\ref{for:megl1}) with $H(\vec{w})=\sum_{k}|w_{k,1}|+|w_{k,2}|$ with the constraint $\diver\vec{J}=0$\label{div0nojoint}

\item problem (\ref{for:megl1}) with $H(\vec{w})=\sum_{k}\max\{|w_{k,1}|,|w_{k,2}|\}$ without the constraint $\diver\vec{J}=0$\label{nodiv0joint}

\item problem (\ref{for:megl1}) with $H(\vec{w})=\sum_{k}\max\{|w_{k,1}|,|w_{k,2}|\}$ with the constraint $\diver\vec{J}=0$\label{div0joint}

\end{enumerate}
(the sum over $k$ is over all voxels). We want to remark that none of the four penalties considered here are rotationally invariant

Another possibility to reconstruct $\vec{J}$ under the constraint $\diver\vec{J}=0$, is to take $\vec{J}=\curl{(G\vec{1}_r)}$, with $G$ a scalar field and to reconstruct a sparse $G$. However, even if the reconstructed $G$ has a small reconstruction error, the error on $\vec{J}$ can be much larger as a result of taking the curl.

Problems (\ref{nodiv0nojoint}) and (\ref{div0nojoint}) can be solved using algorithm (\ref{constrl1alg}) with $\mathbb{S}_{\lambda}$ given componentwise in expression (\ref{Sdef}). As there is no constraint involved in problem (\ref{nodiv0nojoint}), (\ref{nodiv0nojoint}) can be solved more efficiently with the accelerated soft-thresholding algorithm FISTA \cite{BeT:2009}. We use $2000$ iterations. The problem (\ref{div0nojoint}) is solved with $20000$ iterations of algorithm (\ref{constrl1alg}).

For method (\ref{nodiv0joint}) and (\ref{div0joint}), joint sparsity is assumed in the wavelet basis. We therefore simply replace the function $S_\lambda$ used in the algorithms for problems (\ref{nodiv0nojoint}) and (\ref{div0nojoint}) by the nonlinear operator $T_\lambda$ defined in formula (\ref{Tdef}). In other words, we use respectively $2000$ iterations of the FISTA algorithm for problem (\ref{nodiv0joint}) and $20000$ iterations of algorithm (\ref{constrl1alg}) for problem (\ref{div0joint}).

In each of the four methods, the penalty parameter $\lambda$ is chosen to fit the data to the noise level ($\|K\vec{J}_{\mathrm{rec}}-y\|=\|\epsilon\|$).  The results are displayed in Figure~\ref{fig:MEGout}. It should be emphasized that the four reconstructions minimize different functionals (with or without additional constraint). The four reconstructions are therefore not identical, not even in the limit of infinitely many iterations (four different limits).

\begin{figure}
\centering\includegraphics[width=12cm]{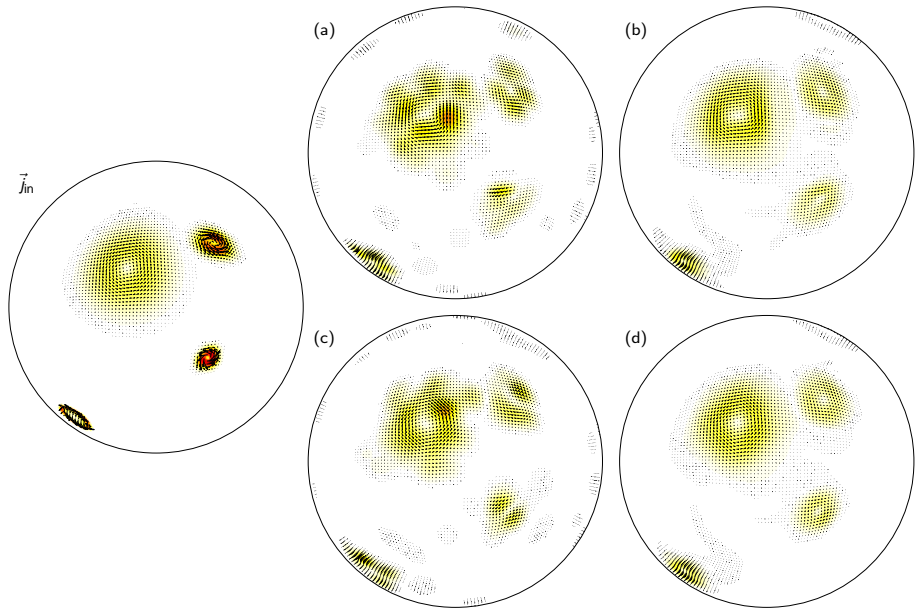}
\caption{Input model $\vec{J}_\mathrm{in}$ and four different reconstructions $\vec{J}_\mathrm{rec}$. Each current density reconstruction is the minimizer of its own functional. They are numbered according to the list in Section~\ref{MEGsection}. The color scale  is the same as in Figure~\ref{fig:MEGin}.}
\label{fig:MEGout}
\end{figure}

When comparing the results, we first conclude that the FISTA algorithm converges faster to its fixed-point than that algorithm (\ref{constrl1alg}) converges to its fixed-point. This is shown in  Figure~\ref{fig:MEGspeed} for cases (\ref{nodiv0nojoint}) and (\ref{div0nojoint}). This is to be expected as FISTA is an accelerated algorithm and  (\ref{constrl1alg}) is not.

\begin{figure}
\centering\includegraphics[width=12cm]{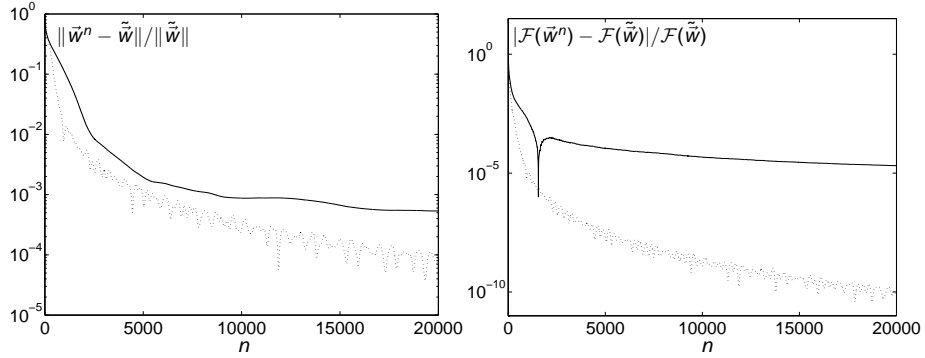}
\caption{Convergence speed of algorithm (\ref{constrl1alg}) (solid) and of FISTA (dotted) to their respective limits as a function of the number of iterations for the problems (\ref{nodiv0nojoint}) and (\ref{div0nojoint}). Left: Relative distances to the respective limits of the algorithms. Right: Difference of the functionals to the respective final values. In both cases the fixed point $\tilde{\vec{w}}$ of the iterations was obtained using $500000$ iterations in their respective algorithms. In case of constrained minimization, the functional value may go below the final limit value. This is due to the fact that the iterates do not satisfy the divergence-free constraint at every step. Again, the two algorithms solve a different minimization problem.}
\label{fig:MEGspeed}
\end{figure}

We also calculate the relative reconstruction error $e_\mathrm{rec}=\|\vec{J}_{\mathrm{in}}-\vec{J}_{\mathrm{rec}} \|/\|\vec{J}_{\mathrm{in}}\|$,
in the four cases (\ref{nodiv0nojoint})--(\ref{div0joint}). The results are displayed in Table~\ref{resulttable}.
The methods that take the constraint into account perform somewhat better than the ones that do not take the constraint into account. Although the difference is quite small, it does not seem to be a fluctuation: the same result was obtained for other input models, in the same wavelet basis as well as in other bases.

\begin{table}\centering
\begin{tabular}{l|cccc}
& \ref{nodiv0nojoint} & \ref{div0nojoint} & \ref{nodiv0joint} & \ref{div0joint} \\  \hline
$e_\mathrm{rec}$ & 0.81 & 0.75 & 0.78 & 0.71\\
$\|\diver \vec{J}_\mathrm{rec}\|$ & $0.59$ & $3.9\cdot 10^{-6}$ & $0.61$ & $8.3\cdot 10^{-6}$\\
$\#$ nonzero in $\vec{w}_\mathrm{rec}$ & $82$ & $9780$ & $194$ & $18108$
\end{tabular}
\caption{Reconstruction data for the four simulations discussed in Section~\ref{MEGsection}.}
\label{resulttable}
\end{table}

The values of $\|\diver \vec{J}\|$ for the four reconstructions are also reported in Table~\ref{resulttable}.
As expected, the constraint is far better satisfied when using methods (\ref{div0nojoint}) and (\ref{div0joint}) with algorithm (\ref{constrl1alg}), than methods (\ref{nodiv0nojoint}) and (\ref{nodiv0joint}).
One of the consequences is that the methods (\ref{div0nojoint}) and (\ref{div0joint}) give better visual results, as can be observed in Figure~{\ref{fig:MEGout}. While all four methods localize the objects in $\vec{J}$ quite well, only those corresponding to divergence-free reconstructions are well structured.

Finally, the divergence-free reconstructions are less sparse than the reconstructions that do not satisfy this constraint. For example when solving case (\ref{nodiv0nojoint}) with the FISTA algorithm, only $82$ of the $49152$ coefficients (in the wavelet basis) are nonzero, while case (\ref{div0nojoint}) solved with algorithm (\ref{constrl1alg}), gives $9780$ nonzero coefficients. This is due to the high number of linear constraints (one for every voxel).

\section{Acknowledgements}
I.L. is research associate of the Fonds de la recherche Scientifique-FNRS (Belgium).
Part of this research was done while the authors were at the Computational and Applied Mathematics Programme of the Vrije Universiteit Brussel and was supported by
VUB GOA-062 and by the Fonds voor Wetenschappelijk Onderzoek-Vlaanderen grant G.0564.09N. The authors would like to thank M.~Fornasier and F.~Pitolli for the useful discussions on MEG and the referees for their constructive comments.

%\section*{References}

\bibliography{constrl1}{}
\bibliographystyle{abbrv}

\end{document}